\documentclass{ws-fml}

\usepackage{}
\usepackage{graphicx}
\usepackage{subfigure}
\usepackage{amsmath}
\usepackage{amssymb}
\usepackage{color}
\usepackage{epstopdf}
\usepackage{multicol}

\input amssym.def
\newsymbol\rtimes 226F
\newfont{\nset}{msbm10}

\begin{document}

\markboth{Z. Xie, Y. Wang, W. Xu, W. Zhu, W. Li \& Z. Zhang}
{Combinatorial properties for a class of simplicial complexes extended from pseudo-fractal scale-free web}

\title{COMBINATORIAL PROPERTIES FOR A CLASS OF SIMPLICIAL COMPLEXES EXTENDED FROM PSEUDO-FRACTAL SCALE-FREE WEB }

\author{ZIXUAN XIE\textsuperscript{1,2}, YUCHENG WANG\textsuperscript{1,3}, WANYUE XU\textsuperscript{1,3},  LIWANG ZHU\textsuperscript{1,3}, WEI LI\textsuperscript{4} and ZHONGZHI ZHANG\textsuperscript{1,3}}
\address{{\textsuperscript{1}Shanghai Key Laboratory of Intelligent Information
Processing, Fudan University, Shanghai 200433, China}\\
{\textsuperscript{2}School of Software, Fudan University, Shanghai 200433, China}\\
{\textsuperscript{3}School of Computer Science, Fudan University, Shanghai 200433, China}\\
{\textsuperscript{4}Academy for Engineering and Technology, Fudan University, Shanghai, 200433, China}\\
\email{\{20302010061,15307130038,xuwy,19210240147,fd\_liwei,zhangzz\}@fudan.edu.cn}}

\maketitle

\footnotetext[1]{Corresponding author: Wei Li and Zhongzhi Zhang.}

\begin{abstract}
Simplicial complexes are a popular tool used to model higher-order interactions between elements of complex social and biological systems. In this paper, we study some combinatorial aspects of a class of simplicial complexes created by a graph product, which is an extension of the pseudo-fractal scale-free web. We determine explicitly the independence number, the domination number, and the chromatic number. Moreover, we derive closed-form expressions for the number of acyclic orientations, the number of root-connected acyclic orientations, the number of spanning trees, as well as the number of perfect matchings for some particular cases.
\end{abstract}

\newpage

\begin{multicols}{2}

\section{INTRODUCTION}

Complex networks have become a popular and powerful formalism for describing diverse types of real-world complex interactive systems in nature and society, whose nodes and edges represent, respectively, the elements and their interactions in real systems~\cite{Ne03}.  In a large majority of previous studies~\cite{Ba16}, the authors consider only pairwise interactions between elements in complex systems, overlooking other interactions such as higher-order ones among multiple elements. Some recent works~\cite{BeGlLe16,GrBaMiAl17,BeAbScJaKl18,SaCaDaLa18,LaPeBaLa19} demonstrate that many real-life systems involve not only dyadic interactions but also interactions among more than two elements at a time. Such multi-way interactions among elements are usually called higher-order interactions or simplicial interactions. For example, in a scientific collaboration network~\cite{PaPeVa17}, for a paper with more than two authors, the interactions among the authors are not pairwise but higher-order. Similar higher-order interactions are also ubiquitous in neuronal spiking activities~\cite{GiPaCuIt15,ReNoScect17}, proteins~\cite{WuOtBa03}, and other real-life systems.

Since the function and various dynamics of a complex system rely to a large extent on the way of interactions between its elements, it is expected that higher-order interactions have a substantial impact on collective dynamics of complex systems with simplicial structure. The last several years have seen some important progress about profound influences of higher-order interactions on different dynamical processes~\cite{BaCeLaetc21}, including percolation~\cite{BiZi18}, public goods game~\cite{AlBadeMoPeLa21}, synchronization~\cite{SkAr19,GaDiGa21}, and epidemic spreading~\cite{MaGoAr20}. For example, in comparison with pairwise interactions, three-way interactions can lead to many novel phenomena, such as Berezinkii-Kosterlitz-Thouless percolation transition~\cite{BiZi18}, abrupt desynchronization~\cite{SkAr19}, as well as abrupt phase transition of epidemic spreading~\cite{MaGoAr20}.

In order to describe the widespread higher-order interactions observed in various real-world complex systems, a lot of models have been proposed~\cite{BaCeLaetc21}, based on some popular mathematical tools, such as simplicial complexes~\cite{CoBi17,PeBa18,LaPeBaLa19,KoSeKh21}.   
\textcolor{red}{A simplex of dimension $d$, called $d$-simplex, represents a single high-order interaction among $d+1$ nodes~\cite{Ha02}, which can be described by a complete graph of $d+1$ nodes. For example, a $0$-simplex is a node, a $1$-simplex is a link, a $3$-simplex is a triangle, while a $4$-simplex is a tetrahedron. For a $d$-simplex $\alpha$, a $\delta$-dimensional face $\alpha'$ of $\alpha$ is a $\delta$-simplex with $0 \leqslant \delta <d$ formed by a subset of the nodes in $\alpha$, i.e., $\alpha' \subseteq \alpha$. For instance, the faces of a $4$-simplex include four nodes, six links, and four triangles. A simplicial complex is a collection of simplices, which is formed by simplices glued along their faces. A simplicial complex is called $d$-dimensional if its constituent simplices are those of dimension at most $d$. Thus, simplicial complexes describe higher-order interactions in a natural way. }

Most of existing models are stochastically, which makes it a challengeable task to exactly analyze their topological and dynamical properties. Very recently, leveraging the edge corona product of graphs~\cite{HaLa93,HaLa95}, a family of iteratively growing deterministic  network model was developed~\cite{WaYiXuZh21,ZhXuZhKaCh22} to describe higher-order interactions. They are called deterministic simplicial  networks, since they are consist of simplexes. This network family subsumes the \textcolor{red}{pseudo-fractal scale-free web~\cite{DoGoMe02} as a particular case, which has received considerable attention from the scientific community, including   fractals~\cite{ZhZh21,XiYu22,WaSlYuZh22}, physics~\cite{RoHaBe07,ZhRoZh07,ZhQiZhXiGu09,ZhLiWuZh10,PeAgZh15,DiBoBe22}, and cybernetics~\cite{YiZhPa20,XuWuZhZhKaCh22}}.   The deterministic construction allow to study exactly at least analytically relevant properties: They display the remarkable scale-free~\cite{BaAl99} and small-world~\cite{WaSt98} properties that are observed in most real-world networks~\cite{Ne03}, and all the eigenvalues and their multiplicities of their normalized Laplacian matrices can be exactly determined~\cite{WaYiXuZh21}.

Although some structural and algebraic properties for the deterministic simplicial networks have been studied, their combinatorial properties  are less explored or not well understood. In this paper, we present an in-depth study on several combinatorial problems for deterministic simplicial networks. \textcolor{red}{Our main contributions are as follows. We first provide an alternative construction of the networks, which shows that the networks are self-similar. We then determine explicitly the domination number, the independence number, as well as the chromatic matching number. Finally, we provide exact formulas for the number of acyclic orientations, the number of root-connected acyclic orientations, the number of spanning trees, and the number of perfect matchings for some special cases. Our exact formulae for the independence number, the domination number, and the number of spanning trees generalize the results~\cite{ShLiZh17,ShLiZh18,ZhLiWuZh10} previously obtained for the pseudo-fractal scale-free web.}

The main reasons for studying the above combinatorial problems lie in at least two aspects. The first one is their inherent theoretical interest~\cite{Yu13,HoKlLiLiPoWa15,GaHaK15,CoLeLi15}, because it is a theoretical challenge to solve these problems. For example, counting all perfect matchings in a graph is \#P-complete~\cite{Va79TCS,Va79SiamJComput}. In view of the hardness, Lov{\'a}sz~\cite{LoPl86} pointed out that it is of great interest to construct or find special graphs for which these combinatorial problems can be exactly solved. The deterministic simplicial networks are in such graph category.
The other justification lies in the relevance of the studied combinatorial problems to practical applications. For example, minimum dominating sets~\cite{NaAk12} and maximum matchings~\cite{LiSlBa11} can be applied to study structural controllability of networks~\cite{LiBa16}, while maximum independent set problem is closely related to graph data mining~\cite{LiLuYaXiWe15,ChLiZh17}. Thus, our work provides useful insight into understanding higher-order structures in the application scenarios of these combinatorial problems.

\section{NETWORK CONSTRUCTIONS AND PROPERTIES}

The family of simplicial networks under consideration was proposed in~\cite{WaYiXuZh21}, which   is constructed based on the edge corona product of graphs~\cite{HaLa93,HaLa95}.  Let $\mathcal{G}_{1}$ and $\mathcal{G}_{2}$ be two graphs with disjoint node sets, where  $\mathcal{G}_{1}$ have $n_1$ nodes and $m_1$ edges. The  edge corona product $\mathcal{G}_{1}  \circledcirc \mathcal{G}_{2}$ of  $\mathcal{G}_{1}$ and $\mathcal{G}_{2}$  is a  graph obtained by taking one replica of $\mathcal{G}_{1}$ and $m_1$ replicas of $\mathcal{G}_{2}$, and
connecting both end nodes of the $i$-th edge of  $\mathcal{G}_{1}$ to every node in the $i$-th replica
of  $\mathcal{G}_{2}$ for $i=1,2,\ldots, m_1$. Let $\mathcal{K}_q$ ($q\geqslant1$) denote the $q$-node complete graph, with $\mathcal{K}_1$ being a graph with an isolate node.  Let $\mathcal{G}_{q}(g)=(\mathcal{V}(\mathcal{G}_{q}(g)),\mathcal{E}(\mathcal{G}_{q}(g)))$  denote the studied networks after $g$ iterations, where $\mathcal{V}(\mathcal{G}_{q}(g))$ and $\mathcal{E}(\mathcal{G}_{q}(g))$ are the sets of nodes and edges, respectively. Then, $\mathcal{G}_q(g)$ is constructed as follows, controlled by two parameters $q$ and $g$ with $q\geqslant1$ and $g \geqslant 0$.

\begin{definition}\label{defa}
For $g=0$, $\mathcal{G}_q(0)$ is the complete graph $\mathcal{K}_{q+2}$. For $g\geqslant 0$, $\mathcal{G}_q(g+1)$ is obtained from $\mathcal{G}_q(g)$  by performing the following operation: for every
existing edge of $\mathcal{G}_q(g)$, one creates a copy of the
complete graph $\mathcal{K}_q$  and connects all its $q$ nodes to both
end nodes of the edge. That is,  $\mathcal{G}_{q}(g+1)= \mathcal{G}_{q}(g) \circledcirc \mathcal{K}_{q}$.
\end{definition}

Figure~\ref{build} illustrates the operation obtaining $\mathcal{G}_q(g+1)$  from $\mathcal{G}_q(g)$, while  Fig.~\ref{net-ex}  illustrates the  network construction processes  for  two cases of $q =1$ and $q =2$.

\begin{figurehere}
\centerline{
\includegraphics[width=0.9\linewidth]{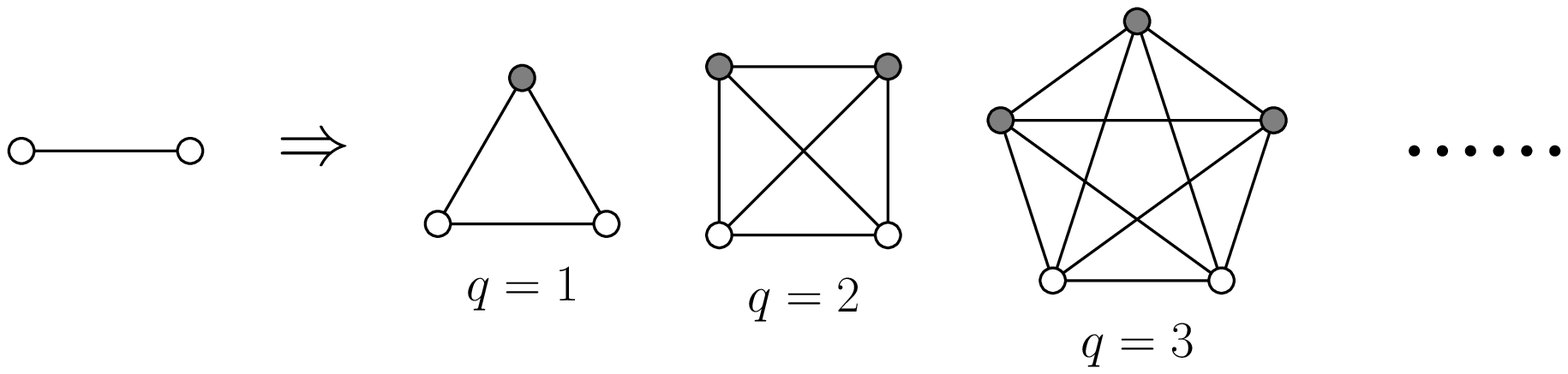}
}
\caption{Network construction approach. \textcolor{red}{For each existing edge in network $\mathcal{G}_q(g)$, performing the operation on  the right-hand side of the arrow generates network $\mathcal{G}_q(g+1)$. The filled circles stand for the nodes constructing the complete graph $\mathcal{K}_q$, and all the filled circles which appeared in step $g+1$ link to both end open nodes of the edge that already exist in the previous step $g$.}}
	\label{build}
\end{figurehere}

\begin{figurehere}
\centerline{
\includegraphics[width=.9\linewidth]{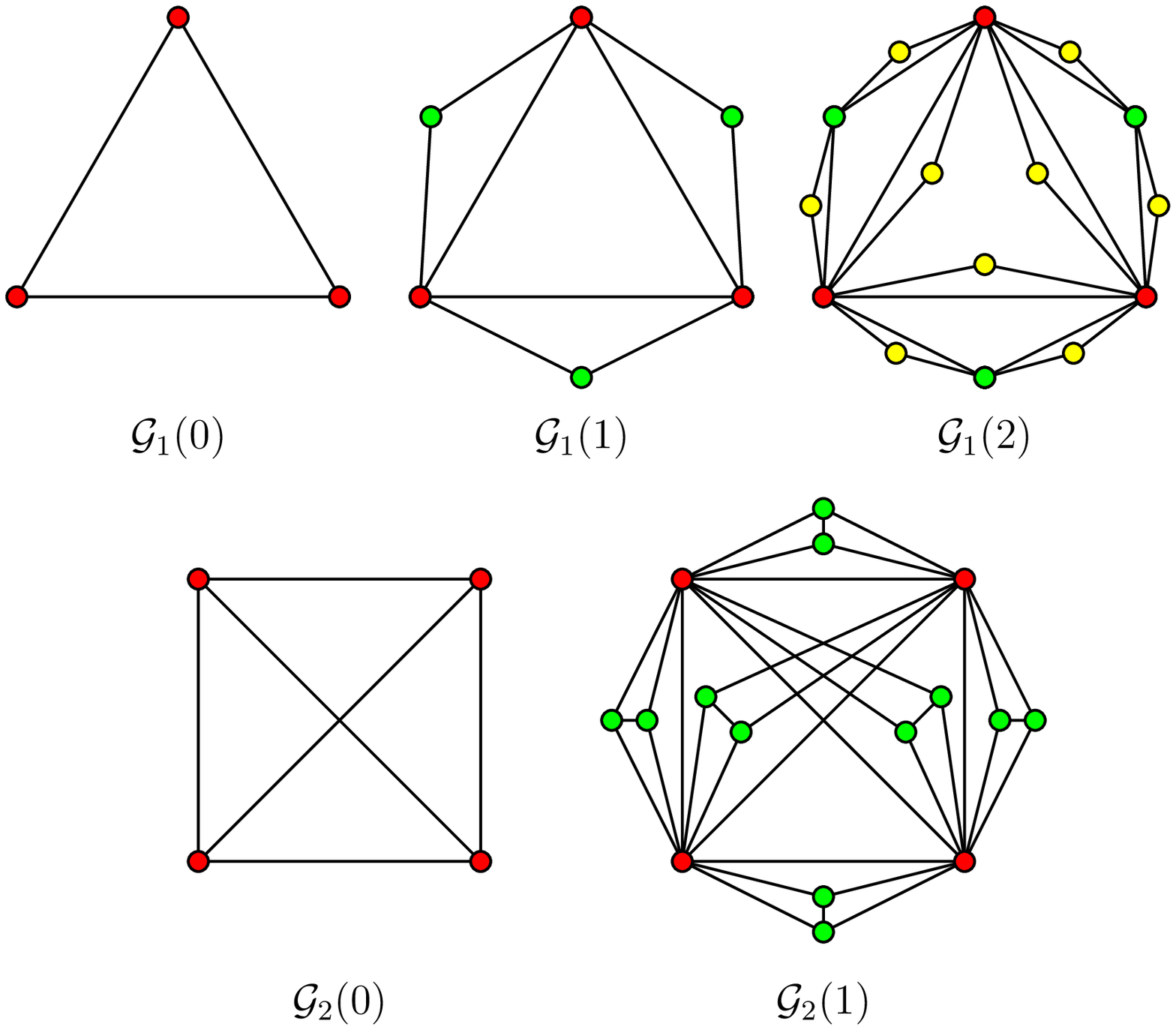}
}
\caption{The first several iterations of $\mathcal{G}_{q}(g)$ for  $q=1$ and $2$. The nodes generated at different iterations are marked with different colors.}
\label{net-ex}
\end{figurehere}

Let $N_{g,q}=|\mathcal{V}(\mathcal{G}_{q}(g))|$ and $M_{g,q}=|\mathcal{E}(\mathcal{G}_{q}(g))|$ denote, respectively, the  number of nodes and the number of edges in $\mathcal{G}_{q}(g)$. By construction, one obtains the following recursion relations for $M_{g,q}$ and $N_{g,q}$:
\begin{equation}
M_{g+1,q}=\frac{(q+1)(q+2)}{2}M_{g,q}
\end{equation}
and
\begin{equation}
N_{g+1,q}=q\,M_{g,q}+N_{g,q},
\end{equation}
which, together with $N_{0,q}=q+2$ and $M_{0,q}=(q+1)(q+2)/2$,   lead to
\begin{equation}\label{eqm}
M_{g,q}={\left[\frac{(q+1)(q+2)}{2}\right]}^{g+1}
\end{equation}
and
\begin{equation}\label{eqn}
N_{g,q}=\frac{2}{q+3}{\left[\frac{(q+1)(q+2)}{2}\right]}^{g+1}+\frac{2(q+2)}{q+3}.
\end{equation}
Thus, the average  degree of nodes in  graph $\mathcal{G}_q(g)$ is $2M_{g,q}/N_{g,q}$, which tends to $q + 3$ when $g$ is large, implying that $\mathcal{G}_q(g)$ is sparse.

For graph $\mathcal{G}_{q}(g)$,  let $\mathcal{W}_{q}(g)=\mathcal{V}(\mathcal{G}_{q}(g))\backslash \mathcal{V}(\mathcal{G}_{q}(g-1))$ represent the set of new nodes generated at iteration $g$. Then,
\begin{equation}\label{W}
	|\mathcal{W}_{q}(g)|=q\left[\frac{(q+1)(q+2)}{2}\right]^{g}.
\end{equation}
Let $d^{(g,q)}_v$ denote the degree of  node $v$ in graph $\mathcal{G}_q(g)$,  which was created at iteration $g_v$. Then,
$d^{(g,q)}_v=(q+1)^{g-g_v+1}$. It is easy to verify that in graph $\mathcal{G}_{q}(g)$, there are $q+2$  nodes with  degree $(q+1)^{g+1}$ and $q\left[\frac{(q+1)(q+2)}{2}\right]^{t_v}$ nodes with degree  $(q+1)^{g-g_v+1}$ for  $0< t_v< g$.

In  graph $\mathcal{G}_q(g)$,  the $q+2$  nodes with the highest degree $(q+1)^{g+1}$ are called hub nodes, which are generated at the initial interaction $g=0$. Let  $h_{k}(g)$, $k=1,2,\ldots, q+2$,  denote the $q+2$ hub nodes of graph $\mathcal{G}_q(g)$, and let $\mathcal{V}_{\rm h}(\mathcal{G}_q(g))$ denote the set of these hub nodes, that is, $\mathcal{V}_{\rm h}(\mathcal{G}_q(g))=\{h_1(g), h_2(g),\ldots, h_{q+2}(g)\}$.  Then, the simplicial networks can be generated in  an alternative way, highlighting the self-similarity.
\begin{proposition}\label{defb}
Given graph $\mathcal{G}_{q}(g)$, graph $\mathcal{G}_{q}(g+1)$ can be obtained by joining $\frac{(q+1)(q+2)}{2}$ copies of $\mathcal{G}_{q}(g)$, denoted as $\mathcal{G}_{q}^{(i,j)}(g)$ $(1\leq i<j\leq q+2)$, the $k$-th $(k=1,2,\ldots,q+2)$ hub node of which is denoted by   $h_{k}^{(i,j)}(g)$. Concretely, in the merging process, for each $k=1,2,\ldots,q+2$, the $q+1$ hub nodes $h_{k}^{(1,k)}(g)$, $h_{k}^{(2,k)}(g)$, $\ldots$, $h_{k}^{(k-1,k)}(g)$, $h_{k}^{(k,k+1)}(g)$, $\ldots$, $h_{k}^{(k,q+2)}(g)$ in the corresponding replicas of \textcolor{red}{$\mathcal{G}_{q}(g)$}  are identified as the hub node $h_{k}(g+1)$ of $\mathcal{G}_{q}(g+1)$.
\end{proposition}

\begin{proof}
We prove this proposition by induction on $g$. For $g=0$, the proof is trivial.
For $g>0$, assume that the conclusion holds for $\mathcal{G}_{q}(g)$, i.e: $\mathcal{G}_{q}(g)$ can be obtained as joining $\frac{(q+1)(q+2)}{2}$ copies of  $\mathcal{G}_{q}(g-1)$, which are denoted by $\mathcal{G}_q^{(i,j)}(g-1)$ with $1\leq i<j\leq q+2$.  During the amalgamation process, for each $k\in \{1,2,\ldots,q+2$\}, the $q+1$ hub nodes $h_{k}^{(1,k)}(g-1)$, $h_{k}^{(2,k)}(g-1)$, $\ldots$, $h_{k}^{(k-1,k)}(g-1)$, $h_{k}^{(k,k+1)}(g-1)$, $\ldots$, $h_{k}^{(k,q+2)}(g-1)$ in the corresponding $q+1$ replicas of $\mathcal{G}_{q}(g-1)$  are identified as the hub node $h_{k}(g)$ of $\mathcal{G}_{q}(g)$. For convenient description, we  use
\begin{align}
\mathcal{G}_{q}(g)=&\textbf{J}\big[\mathcal{G}_{q}^{(1,2)}(g-1),\mathcal{G}_{q}^{(1,3)}(g-1),\ldots,\notag \\
&\mathcal{G}_{q}^{(1,q+2)}(g-1),\ldots,\mathcal{G}_{q}^{(q+1,q-1)}(g-1),\notag\\
&\mathcal{G}_{q}^{(q+1,q+2)}(g-1);\mathcal{V}_{\rm h}(\mathcal{G}_{q}(g))\big],\notag
\end{align}
to denote the above process merging  $\mathcal{G}_q^{(i,j)}(g-1)$ ($1\leq i<j\leq q+2$) to $\mathcal{G}_{q}(g)$, where $\mathcal{V}_{\rm h}(\mathcal{G}_{q}(g)))$ is the set of the $q+2$ hub nodes in $\mathcal{G}_{q}(g)$, which are identified from the hub nodes in the $\frac{(q+1)(q+2)}{2}$ copies  $\mathcal{G}_q^{(i,j)}(g-1)$ of  $\mathcal{G}_{q}(g-1)$.

Next we will prove that the conclusion holds for $\mathcal{G}_{q}(g+1)$. Note that for any graph $\mathcal{G}$ with the degree of its hub nodes larger than 1, $\mathcal{V}_{\rm h}(\mathcal{G}\circledcirc\mathcal{K}_q)=\mathcal{V}_{\rm h}(\mathcal{G})$. By Definition~\ref{defa},  $\mathcal{G}_{q}(g+1)=\mathcal{G}_{q}(g)\circledcirc\mathcal{K}_{q}$.
Thus, using the definition of edge corona product and inductive hypothesis, we have
\begin{align}
&\quad \mathcal{G}_{q}(g+1)=\mathcal{G}_{q}(g)\circledcirc\mathcal{K}_{q}\notag\\
&=\textbf{J}\big[\mathcal{G}_{q}^{(1,2)}(g-1),\mathcal{G}_{q}^{(1,3)}(g-1),\ldots,\notag\\ &\quad \mathcal{G}_{q}^{(q+1,q+2)}(g-1);\mathcal{V}_{\rm h}(\mathcal{G}_{q}(g))\big]\circledcirc\mathcal{K}_{q}\notag\\
&=\textbf{J}\big[\mathcal{G}_{q}^{(1,2)}(g-1)\circledcirc\mathcal{K}_{q},\mathcal{G}_{q}^{(1,3)}(g-1)\circledcirc\mathcal{K}_{q},\ldots,\notag \\ &\quad \mathcal{G}_{q}^{(q+1,q+2)}(g-1)\circledcirc\mathcal{K}_{q};\mathcal{V}_{\rm h}(\mathcal{G}_{q}(g))\big]\notag\\
&=\textbf{J}\big[\mathcal{G}_{q}^{(1,2)}(g),\mathcal{G}_{q}^{(1,3)}(g),\ldots,\mathcal{G}_{q}^{(q+1,q+2)}(g);\notag\\ &\quad \mathcal{V}_{\rm h}(\mathcal{G}_{q}(g))\big]\notag\\
&=\textbf{J}\big[\mathcal{G}_{q}^{(1,2)}(g),\mathcal{G}_{q}^{(1,3)}(g),\ldots,\mathcal{G}_{q}^{(q+1,q+2)}(g);\notag\\ &\quad \mathcal{V}_{\rm h}(\mathcal{G}_{q}(g+1))\big].\notag
\label{reductJ}
\end{align}
This finishes the proof.
\end{proof}

Figure~\ref{net-st} illustrates the second construction way of graph $\mathcal{G}_{q}(g+1)$ for $q=1$ and  $q=2$.


\begin{figure*}
\centerline{
\includegraphics[width=.7\linewidth]{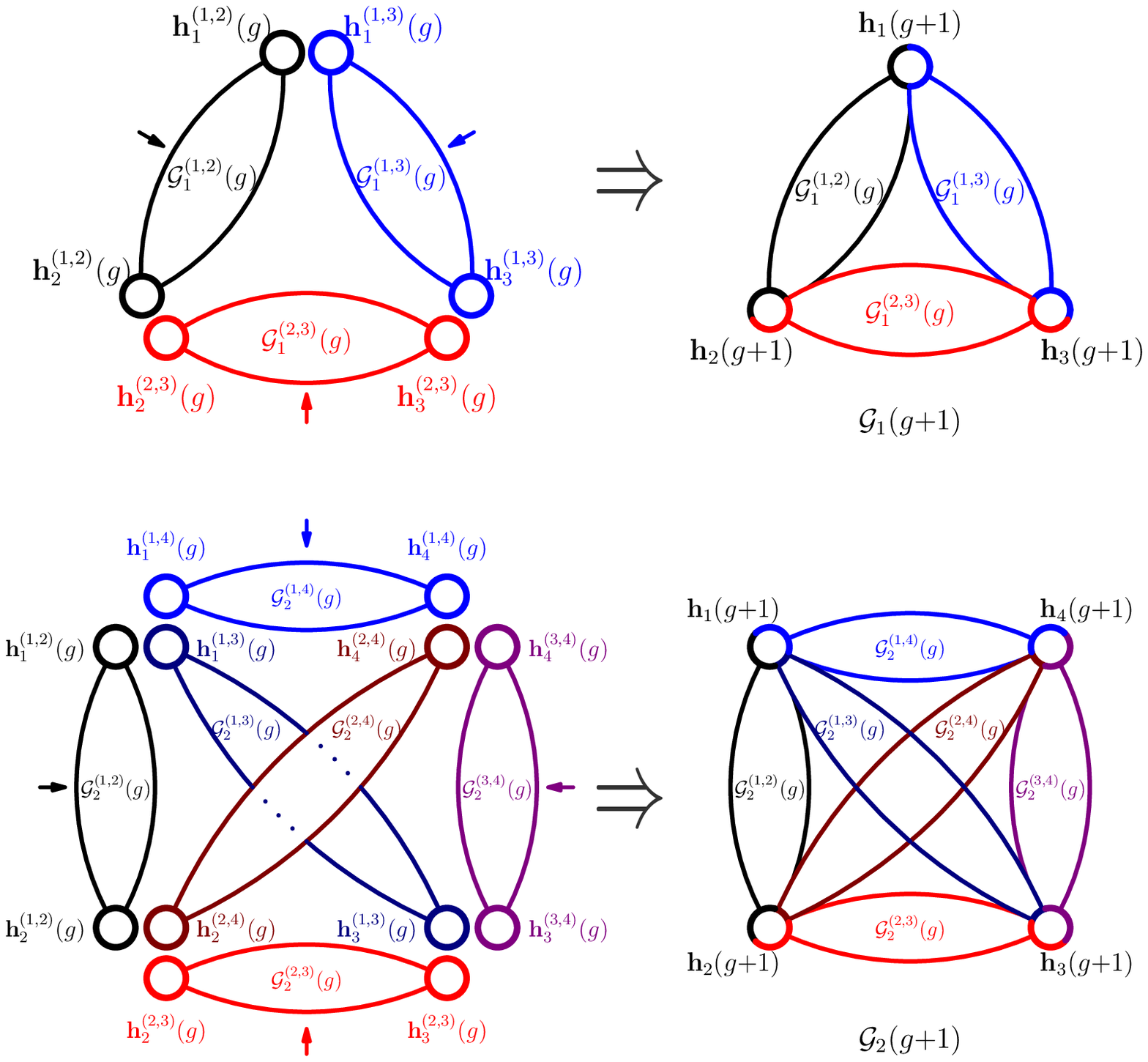}
}
\caption{Second construction means for the simiplical  networks for two special cases of  $q=1$ and   $q=2$. \textcolor{red}{For each $k\in \{1,2,\ldots,q+2$\}, the $q+1$ hub nodes $h_{k}^{(1,k)}(g-1)$, $h_{k}^{(2,k)}(g-1)$, $\ldots$, $h_{k}^{(k-1,k)}(g-1)$, $h_{k}^{(k,k+1)}(g-1)$, $\ldots$, $h_{k}^{(k,q+2)}(g-1)$ in the corresponding $q+1$ replicas of $\mathcal{G}_{q}(g-1)$ are identified as the hub node $h_{k}(g)$ of $\mathcal{G}_{q}(g)$.} }
\label{net-st}
\end{figure*}

The simplicial networks display some remarkable properties~\cite{WaYiXuZh21} as observed in most real networks~\cite{Ne03}. They are scale-free, since their node degrees obey a power-law distribution $P(d)\sim d^{-\gamma_q}$ with $\gamma_q=2 +\frac{\ln (q+2)}{\ln (q+1)}-\frac{\ln 2}{\ln (q+1)}$. They are  small-world, since their diameters grow logarithmically with the number of nodes   and their mean clustering coefficients approach a large constant $\frac{q^2+3q+3 }{q^2+3q+5}$. Moreover, they have a finite spectral dimension $\frac{2[\ln(q^2+3q+3)-\ln 2]}{\ln (q+1)}$.

After introducing the two construction methods  of the simplicial networks  $\mathcal{G}_{q}(g)$ and their relevant properties, in the sequel, we will study analytically  some combinatorial properties of the networks.


\section{INDEPENDENCE NUMBER}

For a simple connected graph  $\mathcal{G}=(\mathcal{V}(\mathcal{G}),\mathcal{E}(\mathcal{G}))$,   abbreviated as $\mathcal{G}=(\mathcal{V},\mathcal{E})$, an independent set of $\mathcal{G}$ is a proper subset $\mathcal{I}$ of $\mathcal{V}$ satisfying that each pair of nodes in $\mathcal{I}$ is not adjacent. An independent set is called a maximal independent set if it is not a subset of any other independent set. A maximal independent set is called a maximum independent set if it has the largest possible cardinality or size. The cardinality of any maximum independent set for a graph $\mathcal{G}$ is called the independence number of $\mathcal{G}$ and is denoted by $\alpha(\mathcal{G})$. We now study the independence number of graph $\mathcal{G}_q(g)$, denoted by $\alpha_q(g)$.
\begin{theorem}
For all $g\geq0$, the independence number of $\mathcal{G}_q(g)$ is:
\begin{equation}\label{alphag}
\alpha_q(g)={\left[\frac{(q+1)(q+2)}{2}\right]}^g.
\end{equation}
\end{theorem}
\begin{proof}
For $g=0$, $\mathcal{G}_q(0)$ is a complete graph of $q+2$ nodes. It is obvious that $\alpha_q(0)=1$, which is consistent with  Eq.~\eqref{alphag}.

For $g\geq1$, by Definition~\ref{defa}, $\mathcal{G}_q(g)$ is obtained from $\mathcal{G}_q(g-1)$ through replacing each of  $M_{g-1,q}$ edges in $\mathcal{G}_q(g-1)$  by a $(q+2)$-node complete graph, which includes the edge and its two end nodes. Let $\mathcal{K}_{q+2}^{(1)}$, $\mathcal{K}_{q+2}^{(2)}$, $\ldots$, $\mathcal{K}_{q+2}^{(M_{g-1,q})}$ denote the $M_{g-1,q}$ complete graphs,  respectively,  corresponding to the $M_{g-1,q}$ edges in   graph  $\mathcal{G}_q(g-1)$.
Then, for any  independent set $\mathcal{I}$ of $\mathcal{G}_q(g)$, there is at most one node in $\mathcal{K}_{q+2}^{(i)}$, $i=1,2,\ldots,M_{g-1,q}$. In other words,
$\left|\mathcal{I}\cap\mathcal{K}_{q+2}^{(i)}\right|\leq1$ for $i=1,2,\ldots,M_{g-1,q}$.
Therefore,  $|\mathcal{I}|\leq M_{g-1,q}={\left[\frac{(q+1)(q+2)}{2}\right]}^g$, implying  $\alpha_q(g)\leq{\left[\frac{(q+1)(q+2)}{2}\right]}^g$.

On the other hand, for every node $u$ in $\mathcal{G}_q(g)$, which is created at generation $g$,  it belongs to a certain clique $\mathcal{K}_{q+2}^{(i)}$ (namely, $u\in\mathcal{K}_{q+2}^{(i)}$), $1\leq i\leq M_{g-1,q}$, and is only connected to the other $q+1$ nodes in this $(q+2)$-clique. Therefore, by arbitrarily selecting one newly created node from each $\mathcal{K}_{q+2}^{(i)}$, $i=1,2,\ldots,M_{g-1,q}$, one obtains an independent set for $\mathcal{G}_q(g)$ with size equal to $M_{g-1,q}={\left[\frac{(q+1)(q+2)}{2}\right]}^g$, which leads to $\alpha_q(g)\geq{\left[\frac{(q+1)(q+2)}{2}\right]}^g$.

Combining the above arguments leads to the statement.
\end{proof}
\textcolor{red}{Equation~\eqref{alphag} generalizes the result in~\cite{ShLiZh18} for $q=1$ to positive integer $q$.}

\section{DOMINATION NUMBER}

For a graph $\mathcal{G}=(\mathcal{V},\mathcal{E})$, a dominating set for $\mathcal{G}$ is a subset $\mathcal{D}$ of $\mathcal{V}$ such that every node not in $\mathcal{D}$ is adjacent to at least one node in $\mathcal{D}$. A dominating set is called a minimal dominating set if  it is not a proper subset of any other dominating set. A dominating set is called a minimum dominating set if it has the smallest cardinality  among all dominating sets. The cardinality of any minimum dominating set for graph $\mathcal{G}$ is  called the domination number of  $\mathcal{G}$, denoted by $\gamma(\mathcal{G})$. For  graph $\mathcal{G}$, the relation $\gamma(\mathcal{G})\leq \alpha(\mathcal{G})$ always holds~\cite{HaPrVo99,WuDuJiLiHu06}.

Let $\gamma_q(g)$ denote the domination number of graph $\mathcal{G}_q(g)$. When $g$ is small, $\gamma_q(g)$ is easily determined. Since $\mathcal{G}_q(0)$ is a complete graph of $q+2$ nodes, $\gamma_q(0)=1$, and every node can be considered as a minimum dominating set for $\mathcal{G}_q(0)$. For $\mathcal{G}_q(1)$, the domination number is obtained in the following lemma.
\begin{lemma}\label{lem}
The domination number of graph $\mathcal{G}_q(1)$ is $\gamma_q(1)=q+1$. And each subset of the hub node set $\mathcal{V}_{\rm h}(\mathcal{G}_q(1))$ containing $q+1$ nodes  is a minimum dominating set for $\mathcal{G}_q(1)$.
\end{lemma}
\begin{proof}
By Proposition~\ref{defb}, $\mathcal{G}_q(1)$ can be generated by joining $\frac{(q+1)(q+2)}{2}$ copies of $\mathcal{K}_{q+2}$ at the $q+2$  hub nodes, which are denoted by $\mathcal{K}^{(i,j)}_{q+2}$, $1\leq i<j\leq q+2$.
Now we show that for  any dominating set $\mathcal{D}$ of $\mathcal{G}_q(1)$, we can construct a  dominating set $\mathcal{D}'$ including only hub nodes in $\mathcal{G}_q(1)$.

Suppose that $v \in \mathcal{D}$
is not a hub node. By  construction,  the neighbors of $v$ are all from a single complete graph $\mathcal{K}^{(i,j)}_{q+2}$. We can replace $v$ by hub node $h_i(1)$ or $h_j(1)$ to obtain a dominating set $\mathcal{D}_{v}$. In a similar way, we can replace other non-hub nodes in $\mathcal{D}_{v}$ to obtain a dominating set $\mathcal{D}'$.  Thus, to find the domination number for $\mathcal{G}_q(1)$,  one can only choose hub nodes to form a minimum dominating set.

We continue to show that for any minimum dominating set $\mathcal{D}$ of  $\mathcal{G}_q(1)$ containing only hub nodes, $|\mathcal{D}|=q+1$. By contradiction, assume that $|\mathcal{D}|<q+1$, which means that there  exist at least  two hub nodes $h_i(1)$ and $h_j(1)$ not in $\mathcal{D}$, which belong to the complete graph $\mathcal{K}^{(i,j)}_{q+2}$. Then, the non-hub nodes in  $\mathcal{K}^{(i,j)}_{q+2}$ are not dominated, implying the  $\mathcal{D}$ is not a dominating set. Therefore, $|\mathcal{D}|\geq q+1$ and $\gamma_q(1)\geq q+1$. On the other hand, it is easy to check that any  $q+1$ hub nodes can dominate all nodes in $\mathcal{G}_{q}(1)$. Hence, $\gamma_q(1)=q+1$.
\end{proof}

Figure~\ref{dom1} illustrates the minimum dominating sets of $\mathcal{G}_1(1)$ and $\mathcal{G}_2(1)$.

\begin{figurehere}
\centerline{
  \includegraphics[width=.7\linewidth]{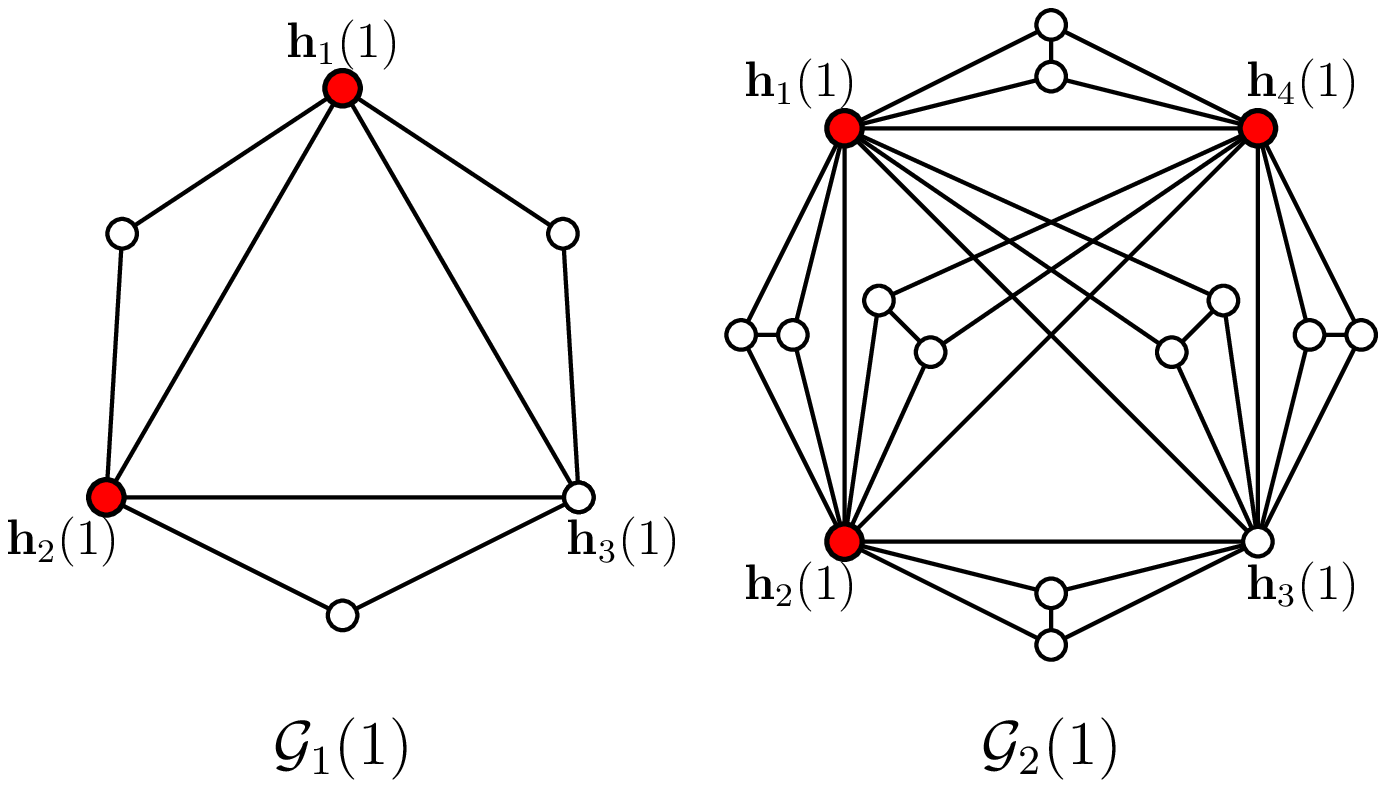}
}
\caption{Examples of the minimum dominating sets for $\mathcal{G}_1(1)$ and $\mathcal{G}_2(1)$.}
\label{dom1}
\end{figurehere}

For a subset $\mathcal{C}$ of $\mathcal{V}(\mathcal{G}_q(g))$, if a node $v$ in $\mathcal{V}(\mathcal{G}_q(g))  \backslash  \mathcal{C}$ is adjacent to at least one node in  $\mathcal{C}$,  we say that $v$ is dominated by $\mathcal{C}$.  Thus, if all nodes in $\mathcal{V}(\mathcal{G}_q(g))  \backslash  \mathcal{C}$ are dominated by $\mathcal{C}$, then $\mathcal{C}$ is a dominating set of $\mathcal{G}_q(g)$.
\begin{lemma}\label{lemD}
For a subset $\mathcal{C}$ of  $\mathcal{V}(\mathcal{G}_q(g))$ corresponding to graph $\mathcal{G}_q(g)$ with $g\geq 1$, if all the non-hub nodes of $\mathcal{G}_q(g)$ are in $\mathcal{C}$ or  dominated by $\mathcal{C}$, then $\mathcal{C}$ is a dominating set of $\mathcal{G}_q(g)$.
\end{lemma}
\begin{proof}
We only need to prove that all the $q+2$ hub nodes of $\mathcal{G}_q(g)$ are either dominated by  $\mathcal{C}$ or included in $\mathcal{C}$. For any hub node $h_i(g)$ ($i=1,2,\ldots,q+2$)  of graph $\mathcal{G}_q(g)$, it and another hub node  $h_j(g)$ ($i \neq j$) create  a clique $\mathcal{K}_q$ with $q$ nodes at generation $g$, which, together with $h_i(g)$ and $h_j(g)$,  form a $(q+2)$-clique  in $\mathcal{G}_q(g)$. If  $h_i(g) \in \mathcal{C}$, the lemma holds. For the case that $h_i(g) \notin \mathcal{C}$, we show below that $h_i(g)$ is dominated by a node in  $\mathcal{C}$. Since the $q$ newly introduced nodes in  $\mathcal{K}_q$ are non-hub nodes in  $\mathcal{G}_q(g)$, for   any node $v \in \mathcal{K}_q$, it  is either included in $\mathcal{C}$ or dominated by  $\mathcal{C}$. If $v \in \mathcal{C} $, then   $h_i(g)$ is dominated by  $\mathcal{C}$. If $v$ is not in $\mathcal{C}$ but dominated by $h_j(g)$ or other non-hub nodes in  its generating $(q+2)$-clique, then $h_i(g)$ is dominated by  $\mathcal{C}$.
 \end{proof}

We are now in position to determine  the domination number $\gamma_q(g)$ of graph $\mathcal{G}_q(g)$ for the case of $g\geq 1$.
\begin{theorem}\label{theoremD}
For all $g\geq1$, the domination number of $\mathcal{G}_q(g)$ is:
\begin{equation*}
 \gamma_q(g)=\frac{q^2+2q-1}{q+3}{\left[\frac{(q+1)(q+2)}{2}\right]}^{g-1}+\frac{2(q+2)}{q+3}.
 \end{equation*}
\end{theorem}
\begin{proof}
From Proposition~\ref{defb},  $\mathcal{G}_q(g+1)$ can be generated by joining $\frac{(q+1)(q+2)}{2}$ copies of $\mathcal{G}_q(g)$ at the $q+2$ hub nodes, denoted by  $\mathcal{G}_{q}^{(1)}(g)$, $\mathcal{G}_{q}^{(2)}(g)$, $\ldots$, $\mathcal{G}_{q}^{(\frac{(q+1)(q+2)}{2})}(g)$, respectively.  Let $\mathcal{D}(\mathcal{G}_q(g+1))$ represent a dominating set of $\mathcal{G}_q(g+1)$. For any $i=1,2,\ldots, \frac{(q+1)(q+2)}{2}$,  the non-hub nodes of $\mathcal{G}_{q}^{(i)}(g)$ are dominated or belong to $\mathcal{D}(\mathcal{G}_q(g+1))\cap\mathcal{G}_{q}^{(i)}(g)$. By Lemma~\ref{lemD}, $\mathcal{D}(\mathcal{G}_q(g+1))\cap\mathcal{G}_{q}^{(i)}(g)$ is  a dominating set of $\mathcal{G}_{q}^{(i)}(g)$.  Then, $|\mathcal{D}(\mathcal{G}_q(g+1))|$ can be computed in terms of $|\mathcal{D}(\mathcal{G}_q(g+1))\cap\mathcal{G}_q^{(i)}(g)|$ as
\begin{align*}
|\mathcal{D}(\mathcal{G}_q(g+1))| &=-q|\mathcal{D}(\mathcal{G}_q(g+1))\cap\mathcal{V}_{\rm h}(\mathcal{G}_q(g+1))|+ \\
&\quad \sum_{i=1}^{\frac{(q+1)(q+2)}{2}}|\mathcal{D}(\mathcal{G}_q(g+1))\cap\mathcal{G}_q^{(i)}(g)|,
\end{align*}
where the first term on the right-hand side compensates for the overcounting of the hub nodes chosen in $\mathcal{D}(\mathcal{G}_q(g+1))$.
Since $|\mathcal{V}_{\rm h}(\mathcal{G}_q(g+1))|=q+2$,
\begin{align*}
|\mathcal{D}(\mathcal{G}_q(g+1))|\geq & \sum_{i=1}^{\frac{(q+1)(q+2)}{2}}|\mathcal{D}(\mathcal{G}_q(g+1))\cap\mathcal{G}_q^{(i)}(g)|\\
&-q(q+2).
\end{align*}
Particularly, when $\mathcal{D}(\mathcal{G}_q(g+1))$ is a minimum dominating set of $\mathcal{G}_q(g+1)$, denoted by $\mathcal{D}_{\rm min}(\mathcal{G}_q(g+1))$,
\begin{align}
&\quad |\mathcal{D}_{\rm min}(\mathcal{G}_q(g+1))| \notag \\ &\geq-q(q+2)+\sum_{i=1}^{\frac{(q+1)(q+2)}{2}}|\mathcal{D}_{\rm min}(\mathcal{G}_q(g+1))\cap\mathcal{G}_q^{(i)}(g)|\nonumber\\
&\geq-q(q+2)+\frac{(q+1)(q+2)}{2}|\mathcal{D}_{\rm min}(\mathcal{G}_q(g))|,\label{eqD}
\end{align}
where the second inequality is due to the fact that $\mathcal{D}_{\rm min}(\mathcal{G}_q(g+1))\cap\mathcal{G}_q^{(i)}(g)$ is a dominating set of $\mathcal{G}_q^{(i)}(g)$, with cardinality larger than that of a minimum dominating set for $\mathcal{G}_q^{(i)}(g)$.

Note that the  terms on both sides of Eq.~\eqref{eqD} are equal to each other,  when all hub nodes of $\mathcal{G}_q(g+1)$ are in  $\mathcal{D}_{\rm min}(\mathcal{G}_q(g+1))$ and the intersection of $\mathcal{D}_{\rm min}(\mathcal{G}_q(g+1))$ with each $\mathcal{G}_q^{(i)}(g)$ forms a minimum dominating set of $\mathcal{G}_q^{(i)}(g)$. In other words, $\mathcal{D}_{\rm min}(\mathcal{G}_q(g+1))$ is the union of the minimum dominating sets $\mathcal{D}^{(i)}_{\rm min}(\mathcal{G}_q(g))$ of  $\mathcal{G}_q^{(i)}(g)$,  such that $\mathcal{D}_{\rm min}(\mathcal{G}_q(g+1))$ includes all hub nodes of $\mathcal{G}_q(g+1)$. Thus, for
 each  $\mathcal{D}^{(i)}_{\rm min}(\mathcal{G}_q(g))$, it contains as many hub nodes in  $\mathcal{G}_q^{(i)}(g)$ as  possible. For $g=1$,  all hub nodes of $\mathcal{G}_q(2)$ and  arbitrary other $q-1$ hub nodes  for each $\mathcal{G}_q^{(i)}(1)$ constitute a minimum dominating set  $\mathcal{D}_{\rm min}(\mathcal{G}_q(2))$ of $\mathcal{G}_q(2)$. By Lemma~\ref{lemD}, $\mathcal{D}_{\rm min}(\mathcal{G}_q(2))\cap\mathcal{G}_q^{(i)}(1)$ forms a minimum dominating set of $\mathcal{G}_q^{(i)}(1)$. It is easy to verify that for $g=1$, the   terms on both sides of Eq.~\eqref{eqD} are equal to each other. Using $\mathcal{D}_{\rm min}(\mathcal{G}_q(2))$, we can construct   a minimum dominating set $\mathcal{D}_{\rm min}(\mathcal{G}_q(3))$ of   $\mathcal{G}_q(3)$ by merging the minimum dominating sets $\mathcal{D}^{(i)}_{\rm min}(\mathcal{G}_q(2))$ for each $\mathcal{G}_q^{(i)}(2)$ and removing those duplicate hub nodes of   $\mathcal{G}_q(3)$. In a similar way, we can iteratively construct   a minimum dominating set $\mathcal{D}_{\rm min}(\mathcal{G}_q(g))$ for $\mathcal{G}_q(g)$ when $g\geq 4$,  for which the equal mark holds for Eq.~\eqref{eqD}.
Then, we have
\begin{equation}
\gamma_q(g+1)=-q(q+2)+\frac{(q+1)(q+2)}{2}\gamma_q(g)
\end{equation}
for all $g\geq1$. With initial condition $\gamma_q(1)=q+1$, the above equation is solved to obtain:
\begin{equation*}
 \gamma_q(g)=\frac{q^2+2q-1}{q+3}{\left[\frac{(q+1)(q+2)}{2}\right]}^{g-1}+\frac{2(q+2)}{q+3}.
 \end{equation*}
This finishes the proof.
\end{proof}
\textcolor{red}{ Theorem~\ref{theoremD} is reduced to the result in~\cite{ShLiZh17} obtained for $q=1$.}

\section{CHROMATIC NUMBER}

Node coloring of a graph  $\mathcal{G}=(\mathcal{V},\mathcal{E})$  is a way of coloring the nodes of $\mathcal{G}$ such that no two adjacent nodes in $\mathcal{V}$ are of the same color. The chromatic number of a graph $\mathcal{G}$, denoted by $\chi(\mathcal{G})$, is the smallest number of colors needed to color the nodes of $\mathcal{G}$. For graph $\mathcal{G}$, node coloring is closely related to its  chromatic polynomial $P(\mathcal{G},\lambda)$, which is a polynomial  counting the number of distinct ways to color $\mathcal{G}$ with $\lambda$ or fewer colors.  The chromatic polynomial was first introduced by George David Birkhoff~\cite{Ge12}.  It contains at least as much information about the colorability of graph $\mathcal{G}$ as does the chromatic number. Indeed, $\chi(\mathcal{G})$ is the smallest positive integer that is not a root of the chromatic polynomial, that is,
 \begin{equation}\label{cnumber}
 \chi(\mathcal{G})={\rm min}\{\lambda:\,P(\mathcal{G},\lambda)>0\}.
 \end{equation}
 The chromatic polynomial for the $q$-node complete graph $\mathcal{K}_q$ is
 \begin{equation}\label{techb}
 P(\mathcal{K}_q,\lambda)=\lambda(\lambda-1)(\lambda-2)\cdots(\lambda-q+1).
 \end{equation}
For two graphs $\mathcal{G}=(\mathcal{V}(\mathcal{G}),\mathcal{E}(\mathcal{G}))$ and  $\mathcal{G'}=(\mathcal{V}(\mathcal{G'}),\mathcal{E}(\mathcal{G'}))$, let $\mathcal{G} \cup\mathcal{G'}$ represent their union with node set $\mathcal{V}(\mathcal{G}) \cup \mathcal{V}(\mathcal{G'})$ and edge set $\mathcal{E}(\mathcal{G}) \cup \mathcal{E}(\mathcal{G'})$, and let $\mathcal{G} \cap\mathcal{G'}$ denote their intersection  with node set $\mathcal{V}(\mathcal{G}) \cap \mathcal{V}(\mathcal{G'})$ and edge set $\mathcal{E}(\mathcal{G}) \cap \mathcal{E}(\mathcal{G'})$. Then, the chromatic polynomial of graph $\mathcal{G}\cup\mathcal{G'}$ is~\cite{Bi93}
 \begin{equation}\label{tech}
 P(\mathcal{G}\cup\mathcal{G'},\lambda)=\frac{P(\mathcal{G},\lambda)\cdot P(\mathcal{G'},\lambda)}{P(\mathcal{G}\cap\mathcal{G'},\lambda)}\,.
 \end{equation}
 \begin{lemma}\label{lemb}
For all $g\geq0$, the chromatic polynomial of $\mathcal{G}_q(g)$ is
\begin{align}\label{chromatic}
& P(\mathcal{G}_q(g),\lambda) \notag\\
 &=\lambda(\lambda-1)\prod_{i=2}^{q+1}{(\lambda-i)}^{\frac{2}{q(q+3)}{\left[\frac{(q+1)(q+2)}{2}\right]}^{g+1}-\frac{2}{q(q+3)}}.
\end{align}
 \end{lemma}
 \begin{proof}
Proposition~\ref{defb} shows that  $\mathcal{G}_q(g+1)$ is in fact an amalgamation of $\frac{(q+1)(q+2)}{2}$ copies of $\mathcal{G}_q(g)$ at the $q+2$ hub nodes, denoted by  $\mathcal{G}_{q}^{(1)}(g)$, $\mathcal{G}_{q}^{(2)}(g)$, $\ldots$, $\mathcal{G}_{q}^{(\frac{(q+1)(q+2)}{2})}(g)$, respectively.  Since the hub nodes of  $\mathcal{G}_q(g+1)$ are linked to each other, $\mathcal{G}_q(g+1)$ can be also obtained from  the $\frac{(q+1)(q+2)}{2}$ copies of $\mathcal{G}_q(g)$ by merging them at the $\frac{(q+1)(q+2)}{2}$ edges of the complete graph $\mathcal{K}_{q+2}$ formed by their $q+2$ hub nodes. That is,
\begin{equation*}
\mathcal{G}_q(g+1)=\mathcal{K}_{q+2}\cup\mathcal{G}_q^{(1)}(g)\cup\mathcal{G}_q^{(2)}(g)\cup\ldots\cup\mathcal{G}_q^{\frac{(q+1)(q+2)}{2}}(g),
\end{equation*}
where
\begin{equation*}
\mathcal{K}_{q+2}\cap\mathcal{G}_q^{(1)}(g)=\mathcal{K}_2
\end{equation*}
and
\begin{equation*}
\bigg(\mathcal{K}_{q+2}\cup\mathcal{G}_q^{(1)}(g)\cup\mathcal{G}_q^{(2)}(g)\ldots\cup\mathcal{G}_q^{(j)}(g)\bigg)\cap\mathcal{G}_q^{(j+1)}(g)=\mathcal{K}_2,
\end{equation*}
for all $1\leq j<\frac{(q+1)(q+2)}{2}$. Thus, by using Eq.~\eqref{tech} $\frac{(q+1)(q+2)}{2}$ times, we  establish a recursion relation between $P(\mathcal{G}_q(g+1),\lambda)$ and $P(\mathcal{G}_q(g),\lambda)$ as
 \begin{equation}\label{ChomaticA}
   P(\mathcal{G}_q(g+1),\lambda)=\frac{ {\left(P(\mathcal{G}_q(g),\lambda)\right)}^{\frac{(q+1)(q+2)}{2}}\cdot P(\mathcal{K}_{q+2},\lambda)}{{\left(P(\mathcal{K}_2,\lambda)\right)}^{\frac{(q+1)(q+2)}{2}}}.
 \end{equation}
 With the initial condition $P(\mathcal{G}_q(0),\lambda)=P(\mathcal{K}_{q+2},\lambda)$ and Eq.~\eqref{techb}, Eq.~\eqref{ChomaticA} is solved to obtain Eq.~\eqref{chromatic}.
 \end{proof}

Combining Eq.~\eqref{cnumber} and Lemma~\ref{lemb},  one obtains the following theorem.
\begin{theorem}
  For all $g\geq0$, the chromatic number of $\mathcal{G}_q(g)$ is
  $\chi(\mathcal{G}_q(g))=q+2$.
\end{theorem}

\section{ENUMERATION OF ACYCLIC ORIENTATIONS}

For an undirected graph $\mathcal{G}=(\mathcal{V},\mathcal{E})$, \textcolor{red}{ an acyclic orientation of $\mathcal{G}$ is to assign a direction to each edge in $\mathcal{G}$ to make it into a directed acyclic graph~\cite{St73}.} An acyclic orientation of  $\mathcal{G}$ is called an acyclic root-connected orientation  when there exists a distinct root node reachable from every node in $\mathcal{G}$ in the resulting directed graph~\cite{GrZa83}.

This section is devoted to the determination of the number  of acyclic orientations, as well as  the number  of acyclic  root-connected orientations in graph $\mathcal{G}_q(g)$. To achieves this goal, we  resort to the tool of Tutte polynomial~\cite{Tu54}. For graph $\mathcal{G}=(\mathcal{V}(\mathcal{G}),\mathcal{E}(\mathcal{G}))$, its Tutte polynomial $T(\mathcal{G};x,y)$ is defined as
\begin{equation}
T(\mathcal{G};x,y)=\sum_{\mathcal{H}\subseteq\mathcal{G}}{(x-1)}^{r(\mathcal{G})-r(\mathcal{H})}{(y-1)}^{n(\mathcal{H})},
\end{equation}
where the sum runs over all the spanning subgraphs $\mathcal{H}$ of $\mathcal{G}$, $r(\mathcal{G})=|\mathcal{V}(\mathcal{G})|-k(\mathcal{G})$ is the rank of $\mathcal{G}$, $n(\mathcal{G})=|\mathcal{E}(\mathcal{G})|-|\mathcal{V}(\mathcal{G})|+k(\mathcal{G})$ is the nullity of $\mathcal{G}$, and $k(\mathcal{G})$ is the number of components of $\mathcal{G}$.

The evaluation of the Tutte polynomial of graph $\mathcal{G}$ at a particular point on $(x,y)$-plane is related to many combinatorial  aspects of  $\mathcal{G}$~\cite{We99}.  It has been shown that  $T(\mathcal{G};2,0)$ equals the number of acyclic orientations of $\mathcal{G}$~\cite{St73}, while
 $T(\mathcal{G};1,0)$ is equivalent to the number of root-connected acyclic orientations of $\mathcal{G}$~\cite{GrZa83}.
Moreover, the Tutte polynomial $T(\mathcal{G};x,y)$ is also relevant to  the chromatic polynomial $P(\mathcal{G},\lambda)$ of graph  $\mathcal{G}$. Specifically,  $P(\mathcal{G},\lambda)$ can be represented in terms of $T(\mathcal{G};x,y)$ at $y=0$ as
 \begin{equation}\label{relat}
 P(\mathcal{G},\lambda)={(-\lambda)}^{k(\mathcal{G})}{(-1)}^{n(\mathcal{G})}T(\mathcal{G};1-\lambda,0).
 \end{equation}
This connection between the Tutte polynomial and the chromatic polynomial allows to determine the number of acyclic orientations and root-connected acyclic orientations for $\mathcal{G}_q(g)$.
 \begin{theorem}
 For graph $\mathcal{G}_q(g)$ with $g\geq0$, the number of acyclic orientations  is
   \begin{equation}\label{AO}
   N_{\rm ao}(\mathcal{G}_q(g))=2{\left[\frac{(q+2)!}{2}\right]}^{\frac{2}{q(q+3)}{\left[\frac{(q+1)(q+2)}{2}\right]}^{g+1}-\frac{2}{q(q+3)}},
   \end{equation}
and the number of root-connected acyclic orientations  is
   \begin{equation}\label{RAO}
   N_{\rm rao}(\mathcal{G}_q(g))={[(q+1)!]}^{\frac{2}{q(q+3)}{\left[\frac{(q+1)(q+2)}{2}\right]}^{g+1}-\frac{2}{q(q+3)}}.
   \end{equation}
 \end{theorem}
 \begin{proof}
Lemma~\ref{lemb} gives the chromatic polynomial for graph $\mathcal{G}_q(g)$, which together with Eq.~\eqref{relat} governing the Tutte polynomial and the chromatic polynomial yields
 \begin{equation}\label{eqp}
 T(\mathcal{G}_q(g);x,0)=x\prod_{i=1}^{q}{(x+i)}^{\frac{2}{q(q+3)}{\left[\frac{(q+1)(q+2)}{2}\right]}^{g+1}-\frac{2}{q(q+3)}}.
 \end{equation}
By replacing $x=2$ and $x=1$ in Eq.~\eqref{eqp}, we obtain the number of acyclic orientations and
the number of acyclic root-connected orientations of $\mathcal{G}_q(g)$, as given by Eqs.~\eqref{AO} and~\eqref{RAO}, respectively.
 \end{proof}

\section{NUMBER OF PERFECT MATCHINGS}

For  a graph $\mathcal{G}=(\mathcal{V},\mathcal{E})$, a matching $\mathcal{M}$ of $\mathcal{G}$ is a subset of $\mathcal{E}$ such that no two edges in $\mathcal{M}$ share a common node. A node is called matched or covered by a matching, if it is an endpoint of one of the edges in the matching. For a graph $\mathcal{G}$ with even number of nodes, a perfect matching of $\mathcal{G}$  is a matching which matches all nodes in the graph. It was shown that for a complete graph with even $q$ nodes,   the number of its perfect matchings  is $(q-1)!!$~\cite{DiHo98}.

For graph $\mathcal{G}_q(g+1)$ and  $\mathcal{G}_q(g)$, by the second construction approach given in Proposition~\ref{defb}, the number of  their nodes satisfies  $N_{g+1,q}=\frac{(q+1)(q+2)}{2} N_{g,q}-q(q+2)$.  Hence,  $N_{g,q}$ is always even when $q$ is an even number; but  $N_{g,q}$ may be an  odd number when   $q$ is odd. Then,  for odd $q$, a perfect matching may not  exist for $\mathcal{G}_q(g)$. Below, we will show that for an even $q$, perfect matchings always exist in $\mathcal{G}_q(g)$.

\begin{theorem}
When $q$ is even and not less than  2,  perfect matchings  always exist in $\mathcal{G}_q(g)$  for all $g\geq 0$.
\end{theorem}
\begin{proof}
By  induction on $g$. For $g=0$, $\mathcal{G}_q(0)$ is a complete graph with $q+2$ nodes. There exist $(q+1)!!$ perfect matchings  in $\mathcal{G}_q(0)$. Thus, the result is true for $g=0$. Suppose that there is a perfect matchings  in $\mathcal{G}_q(g)$.  Let  $\mathcal{M}_g$ be a perfect matching  of $\mathcal{G}_q(g)$ for $g\geq 0$. By construction in Definition~\ref{defa},  $\mathcal{G}_q(g+1)$ is obtained from $\mathcal{G}_q(g)$ by replacing each of  $M_{g,q}$ edges in $\mathcal{G}_q(g)$  by a complete graph having $(q+2)$ nodes, which includes the edge and its  two end nodes. Let $\mathcal{K}_{q+2}^{(1)}$, $\mathcal{K}_{q+2}^{(2)}$, $\ldots$, $\mathcal{K}_{q+2}^{(M_{g,q})}$ denote the $M_{g,q}$ complete graphs. For each of these complete graphs,  corresponding  to an edge in   $\mathcal{G}_q(g)$, since the two end nodes are covered by  $\mathcal{M}_g$, one can chose $q/2$ independent edges in the complete graph to cover the remaining nodes. In this way, one obtains a  perfect matchings  for $\mathcal{G}_q(g+1)$.
\end{proof}

We proceed to determine  the number of perfect matchings in $\mathcal{G}_q(g)$ for even $q$, denoted by $N_{\rm per}(g,q)$. For this purpose, we first  define some intermediary variables for graph $\mathcal{G}_q(g)$.
Let $\Upsilon_0(\mathcal{G}_q(g))$ denote the set of matchings of $\mathcal{G}_q(g)$ such that the two hub nodes $h_1(g)$ and $h_2(g)$ are vacant, while  all the other nodes in $\mathcal{G}_q(g)$ are matched. Let $A_{q}(g)$ be the cardinality of  $\Upsilon_0(\mathcal{G}_q(g))$. Let $\Upsilon_1(\mathcal{G}_q(g))$ be the set of  perfect matchings of $\mathcal{G}_q(g)$, and let $B_{q}(g)=N_{\rm per}(g,q)$ be the cardinality of  $\Upsilon_1(\mathcal{G}_q(g))$. Note that in $\mathcal{G}_q(g)$, there does not exist such a match that either $h_1(g)$ or $h_2(g)$ is vacant, while all other nodes are covered.

\begin{theorem}\label{theo1}
 For even $q\geq 2$ and $g\geq0$, the number of perfect matchings in $\mathcal{G}_q(g)$ is
  \begin{align}\label{Dimer}
&N_{\rm per}(g,q)={[(q+1)!!]}^{\frac{2}{q(q+3)}{\left[\frac{(q+1)(q+2)}{2}\right]}^{g+1}-\frac{2}{q(q+3)}} \times \notag\\
    &\quad{(q+1)}^{-\frac{2(q+2)}{q{(q+3)}^2}{\left[\frac{(q+1)(q+2)}{2}\right]}^{g+1}+\frac{q+2}{q+3}g+\frac{(q+1){(q+2)}^2}{q{(q+3)}^2}}.
  \end{align}
\end{theorem}
\begin{proof}
Note that $B_{q}(g)=N_{\rm per}(g,q)$. We now derive recursive relations for $A_{q}(g+1)$ and  $B_{q}(g+1)$, on the basis of  which we further determine an exact expression for  $B_{q}(g)$. From Proposition~\ref{defb} and the definitions of   $A_{q}(g)$ and  $B_{q}(g)$, one obtains the following recursion relations for $A_{q}(g)$ and $B_{q}(g)$:
 \begin{equation}\label{eqqa}
  A_{q}(g+1)=(q-1)!!B_{q}^{\frac{q}{2}}(g)A_{q}^{\frac{q^2+2q+2}{2}}(g)\,,
 \end{equation}
 \begin{equation}\label{eqqb}
  B_{q}(g+1)=(q+1)!!B_{q}^{\frac{q}{2}+1}(g)A_{q}^{\frac{q(q+2)}{2}}(g)\,.
 \end{equation}

\textcolor{red}{Equation~\eqref{eqqa} is explained as follows. By Proposition~\ref{defb}, any matching for graph $\mathcal{G}_q(g+1)$ can be constructed by joining the matchings of the $\frac{(q+1)(q+2)}{2}$ copies of $\mathcal{G}_q(g)$ at the $q+2$ hub vertices. According to the aforementioned analysis, for any matching of $\mathcal{G}_q(g+1)$, the matched hub vertices of $\mathcal{G}_q(g+1)$ must come in pairs. Then, any matching  in $\Upsilon_0(\mathcal{G}_q(g+1))$ can be obtained recursively from those in $\Upsilon_0(\mathcal{G}_q(g))$ and $\Upsilon_1(\mathcal{G}_q(g))$ for  the $\frac{(q+1)(q+2)}{2}$ copies of $\mathcal{G}_q(g)$, denoted by $\mathcal{G}_q^{(i)}(g)$, $i=1,2,\ldots, \frac{(q+1)(q+2)}{2}$. Moreover, related quantities in Eq.~\eqref{eqqa} are accounted for as follows: $(q-1)!!$ indicates the ways of pairing the $q$ matched hub vertices, which is equal to the number of perfect matchings of $\mathcal{K}_{q}$; $\frac{q}{2}$ denotes the pairs of matched hub vertices $h_1(g)$ or $h_2(g)$ in $\mathcal{G}_q^{(i)}(g)$; while $\frac{q^2+2q+2}{2}=\frac{(q+1)(q+2)}{2}-\frac{q}{2}$ represents the cases that $h_1(g)$ or $h_2(g)$ are vacant in $\mathcal{G}_q^{(i)}(g)$. Analogously, we can interpret Eq.~\eqref{eqqb}.}

Figure~\ref{match1} gives a graphic illustration of  Eqs.~\eqref{eqqa} and~\eqref{eqqb} for a particular case of $q=2$.

Dividing Eq.~\eqref{eqqb} by Eq.~\eqref{eqqa} yields a recursive relation for $\frac{B_{q}(g)}{A_{q}(g)}$ as
\begin{equation}\label{eqqc}
\frac{B_{q}(g+1)}{A_{q}(g+1)}=(q+1)\frac{B_{q}(g)}{A_{q}(g)}\,.
\end{equation}
With the initial condition $\frac{B_{q}(0)}{A_{q}(0)}=\frac{(q+1)!!}{(q-1)!!}=q+1$, Eq.~\eqref{eqqc} is solved to obtain
\begin{equation}\label{eqqd}
\frac{B_{q}(g)}{A_{q}(g)}={(q+1)}^{g+1},
\end{equation}
which implies
\begin{equation}
A_{q}(g)=\frac{B_{q}(g)}{{(q+1)}^{g+1}}.
\end{equation}
Plugging this expression into Eq.~\eqref{eqqb} leads to
\begin{equation}\label{eqqf}
B_{q}(g+1)=\frac{(q+1)!!}{{(q+1)}^{\frac{q(q+2)(g+1)}{2}}}B_{q}^{\frac{(q+1)(q+2)}{2}}(g)\,.
\end{equation}
Under the initial condition $B_{q}(0)=(q+1)!!$, Eq.~\eqref{eqqf} is solved to obtain Eq.~\eqref{Dimer}.
\end{proof}

\begin{figurehere}
\centerline{
\includegraphics[width=\linewidth]{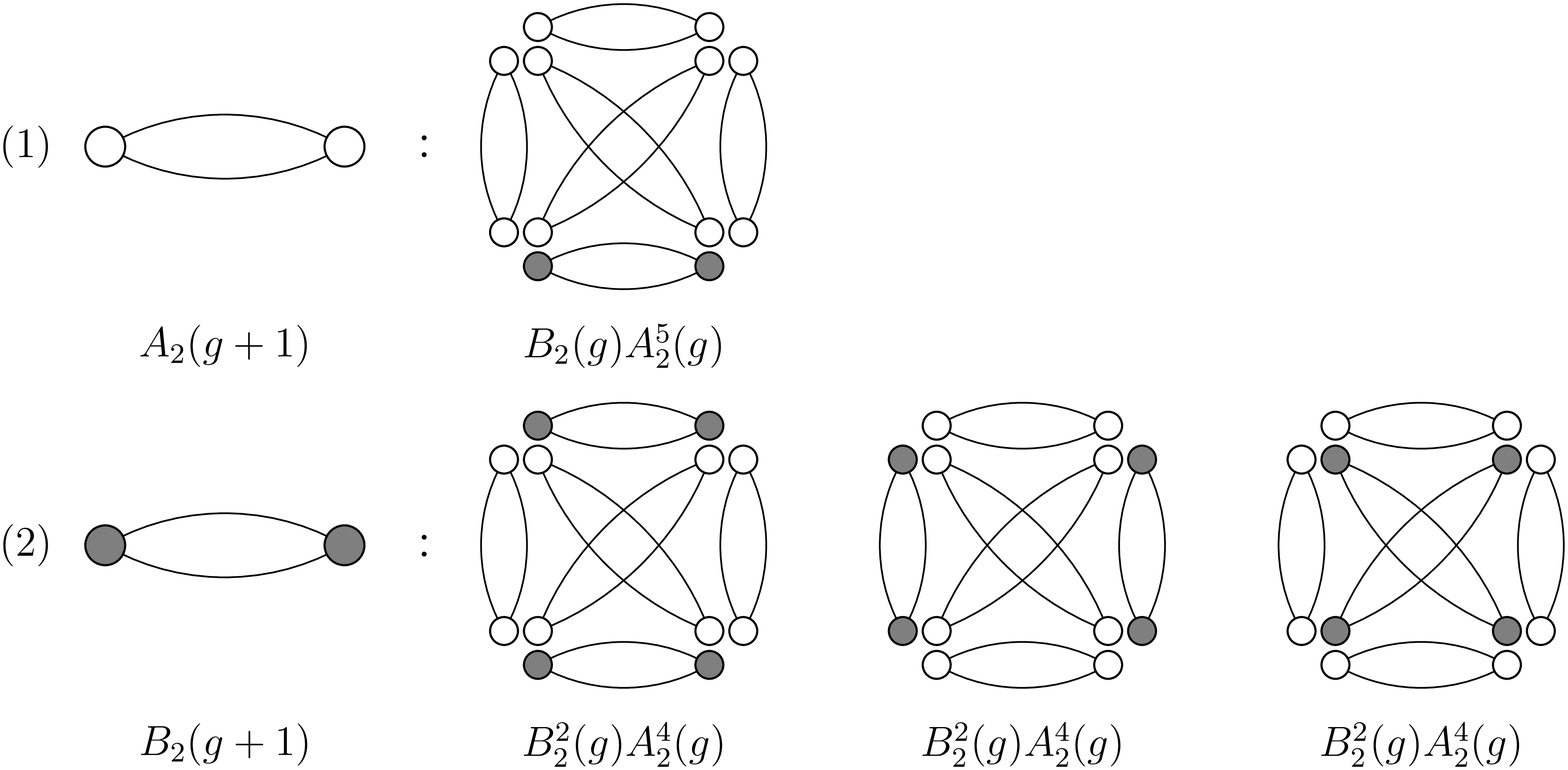}
}
\caption{Illustration of recursive configurations for matchings in $\Upsilon_0(\mathcal{G}_q(g+1))$ and $\Upsilon_1(\mathcal{G}_q(g+1))$ for graph $\mathcal{G}_q(g+1)$ with $q=2$.  Each filled node represents a covered node, while each empty node represents a vacant node. }
\label{match1}
\end{figurehere}

\section{NUMBER OF SPANNING TREES}

For  a graph $\mathcal{G}=(\mathcal{V},\mathcal{E})$ with $|\mathcal{V}|$ nodes, a spanning tree of  $\mathcal{G}$ is a connected subgraph of $\mathcal{G}$ that has a   node set $\mathcal{V}$ and $|\mathcal{V}|-1$ edges. The number of spanning trees in a graph $\mathcal{G}$ is called the complexity of $\mathcal{G}$, which is an important graph invariant~\cite{LiPaYiZh20}.

\begin{lemma}\label{LemNS}
Let $u$ and $v$ denote two distinct nodes in a complete graph $\mathcal{K}_q$  with $q \geq 3$ nodes. Let $N_{u,v}(\mathcal{K}_q) $ be the number of  spanning forests for $\mathcal{K}_q$,  each of which consists two  trees such that   $u$ and $v$ belong to the two different trees. Then, $N_{u,v}(\mathcal{K}_q)=2q^{q-3}$.
\end{lemma}
\begin{proof}
Obviously,  $N_{u,v}(\mathcal{K}_q)$ equals  the number of those spanning trees in $\mathcal{K}_q$,  in each of which $u$ and $v$ are adjacent to each. By Cayley's formula~\cite{Ca89}, the number of spanning trees of a $q$-node complete graph  is $q^{q-2}$. For a randomly chosen  spanning tree $\mathcal{T}$ of $\mathcal{K}_q$,  the expected degree of node $u$ is $\frac{2(q-1)}{q}$. Then, in $\mathcal{T}$ the probability that $u$ is a neighbor of $v$ is $\frac{2(q-1)}{q}\cdot\frac{1}{q-1}=\frac{2}{q}$. Therefore, $N_{u,v}(\mathcal{K}_q)=q^{q-2}\cdot \frac{2}{q}=2q^{q-3}$.
\end{proof}


\begin{theorem}\label{Theoremtrees01}
Let $ N_{\rm st}(g,q)$ be the number of spanning trees in graph $\mathcal{G}_{q}(g)$ with $g\geq0$ and $q\geq 1 $. Then,
 \begin{align}
   &N_{\rm st}(g,q)=2^{\frac{2(q+1)}{q{(q+3)}^2}{\left[\frac{(q+1)(q+2)}{2}\right]}^{g+1}-\left(\frac{q+1}{q+3}\right)g-\frac{{(q+1)}^2(q+2)}{{q(q+3)}^2}}\times
     \notag \\
   &\quad {(q+2)}^{\frac{2(q^2+2q-1)}{q{(q+3)}^2}{\left[\frac{(q+1)(q+2)}{2}\right]}^{g+1}+\left(\frac{q+1}{q+3}\right)g+\frac{q^3+2q^2-q+2}{{q(q+3)}^2}}.\notag
  \end{align}
\end{theorem}
\begin{proof}
By Definition~\ref{defa}, $\mathcal{G}_{q}(g+1)$ is obtained from  $\mathcal{G}_{q}(g)$  through replacing every edge in $\mathcal{G}_{q}(g)$ by a $\mathcal{K}_{q+2}$ clique,  which contains the edge and its two end nodes.  Accordingly, for any spanning tree $\mathcal{T}_g$ for $\mathcal{G}_{q}(g)$, one can construct a spanning tree  for $\mathcal{G}_{q}(g+1)$ in the following way:  replace each edge in $\mathcal{T}_g$ by a spanning tree of the $\mathcal{K}_{q+2}$ clique and replace each edge $(u,v)$ in $\mathcal{G}_{q}(g)$ but absent in $\mathcal{T}_g$ by a spanning forest for $\mathcal{K}_{q+2}$,  which consists two trees such that $u$ and $v$ are in  different
trees. Thus, using Lemma~\ref{LemNS},  one obtains
\begin{align}\label{trees01}
&\quad  N_{\rm st}(g+1,q) \notag\\&=[(q+2)^q]^{N_{g,q}-1}\,[2(q+2)^{q-1}]^{M_{g,q}-N_{g,q}+1}\, N_{\rm st}(g,q)\notag\\
&= 2^{M_{g,q}-N_{g,q}+1}\,{(q+2)}^{(q-1)M_{g,q}+N_{g,q}-1}\, N_{\rm st}(g,q),
\end{align}
\textcolor{red}{where $N_{g,q}-1$ is the number of edges in the spanning tree $\mathcal{T}_g$ in $\mathcal{G}_q(g)$, and $M_{g,q}-N_{g,q}+1$ is the number of edges which exist in $\mathcal{G}_q(g)$ but is absent in $\mathcal{T}_g$, in other words, the number of edges in $\mathcal{G}_q(g)$ minus the number of edges in $\mathcal{T}_g$.}
Plugging the expressions for  $N_{g,q}$ and $M_{g,q}$ in Eqs.~\eqref{eqn} and~\eqref{eqm} into Eq.~\eqref{trees01} gives \begin{align}
&\quad   N_{\rm st}(g+1,q) \notag\\&=
N_{\rm st}(g,q)2^{\frac{q+1}{q+3}{\left[\frac{(q+1)(q+2)}{2}\right]}^{g+1}-\frac{q+1}{q+3}}\notag\\
    &\quad ¡¡\cdot {(q+2)}^{\frac{q^2+2q-1}{q+3}{\left[\frac{(q+1)(q+2)}{2}\right]}^{g+1}+\frac{q+1}{q+3}} .
    \end{align}
With the initial condition $N_{\rm st}(0,q)= {(q+2)}^q$, one obtains
    \begin{align}
  &N_{\rm st}(g,q) \notag\\=& 2^{\frac{q+1}{q+3}\sum_{i=1}^{g}{\left[\frac{(q+1)(q+2)}{2}\right]}^{i}-\frac{q+1}{q+3}g}\times{(q+2)}^q\times\notag\\
    & {(q+2)}^{\frac{q^2+2q-1}{q+3}\sum_{i=1}^{g}{\left[\frac{(q+1)(q+2)}{2}\right]}^{i}+\frac{q+1}{q+3}g}\notag\\
=&2^{\frac{2(q+1)}{q{(q+3)}^2}{\left[\frac{(q+1)(q+2)}{2}\right]}^{g+1}-\left(\frac{q+1}{q+3}\right)g-\frac{{(q+1)}^2(q+2)}{{q(q+3)}^2}}\times
      \notag \\
     &{(q+2)}^{\frac{2(q^2+2q-1)}{q{(q+3)}^2}{\left[\frac{(q+1)(q+2)}{2}\right]}^{g+1}+\left(\frac{q+1}{q+3}\right)g+\frac{q^3+2q^2-q+2}{{q(q+3)}^2}},\notag
    \end{align}
which completes the proof.
\end{proof}

\textcolor{red}{Note that Theorem~\ref{Theoremtrees01} is consistent with result in~\cite{WaYiXuZh21}, obtained using different technique, and reduces to the result~\cite{ZhLiWuZh10} for the special case of $q=1$.}

\section{CONCLUSIONS}

In this paper, we presented a systematic analytical study of combinatorial properties for a class of iteratively generating simplicial networks, based on their particular construction and self-similar structure. We derived the domination number, the independence number, and the chromatic number. Moreover, we obtained exact expressions for the number of spanning trees, the number of perfect matchings for even $q$, the number of acyclic orientations, and the number of root-connected acyclic orientations.

The considered combinatorial problems are a fundamental research subject of theoretical computer science, many of which are NP-hard and even \#P-complete for a general graph. It is thus of great interest to study the special family of graphs for which these challenging combinatorial problems can be exactly solved. In addition, since the considered combinatorial problems are relevant to various practical application in the aspects of network science and graph data miming, this work provides insight into understanding the applications of these combinatorial problems for simplicial complexes.

\textcolor{red}{Finally, it is worth mentioning that the studied simplicial networks are in fact constructed iteratively by edge corona product, which leads to the self-similarity of the resulting graphs. It is expected that our computation approach and process for relevant problems are also applicable to other graph families~\cite{DoMa05,ZhZhZoChGu07,AgBu09,ZhGuDiChZh09,CoFeRa04,BaCoDaFi09,QiYiZh19} with self-similar properties, built by other graph operations. We note that our techniques also have some limitations. For example, they do not apply to recursive graphs with stochasticity~\cite{HiBe06,ZhZhZoChGu09}.}





%



  \section*{ACKNOWLEDGEMENT}

The work was supported by the Shanghai Municipal Science and Technology Major Project  (Nos.  2018SHZDZX01 and 2021SHZDZX0103), the National Natural Science Foundation of China (Nos. 61872093  and  U20B2051),  Ji Hua Laboratory, Foshan, China (No.X190011TB190),  ZJ Lab, and Shanghai Center for Brain Science and Brain-Inspired Technology. Zixuan Xie was also supported by Fudan's Undergraduate Research Opportunities Program (FDUROP) under Grant No. 22099.






%

\bibliographystyle{ws-ijait}
\bibliography{Combination}

\providecommand{\noopsort}[1]{}\providecommand{\singleletter}[1]{#1}%
\begin{thebibliography}{10}

\bibitem{Ne03}
M.~E. Newman, The structure and function of complex networks, {\em SIAM Rev.}
  {\bf 45}(2)  (2003)  167--256.

\bibitem{Ba16}
A.-L. Barab{\'a}si, {\em Network science} (Cambridge University Press, 2016).

\bibitem{BeGlLe16}
A.~R. Benson, D.~F. Gleich and J.~Leskovec, Higher-order organization of
  complex networks, {\em Science} {\bf 353}(6295)  (2016)  163--166.

\bibitem{GrBaMiAl17}
J.~Grilli, G.~Barab{\'a}s, M.~J. Michalska-Smith and S.~Allesina, Higher-order
  interactions stabilize dynamics in competitive network models, {\em Nature}
  {\bf 548}(7666)  (2017) p. 210.

\bibitem{BeAbScJaKl18}
A.~R. Benson, R.~Abebe, M.~T. Schaub, A.~Jadbabaie and J.~Kleinberg, Simplicial
  closure and higher-order link prediction, {\em Proc. Natl. Acad. Sci.} {\bf
  115}(48)  (2018)  E11221--E11230.

\bibitem{SaCaDaLa18}
V.~Salnikov, D.~Cassese and R.~Lambiotte, Simplicial complexes and complex
  systems, {\em Eur. J. Phys.} {\bf 40}(1)  (2018) p. 014001.

\bibitem{LaPeBaLa19}
I.~Lacopini, G.~Petri, A.~Barrat and V.~Latora, Simplicial models of social
  contagion, {\em Nat. Commun.} {\bf 10}(1)  (2019) p. 2485.

\bibitem{PaPeVa17}
A.~Patania, G.~Petri and F.~Vaccarino, The shape of collaborations, {\em EPJ
  Data Sci.} {\bf 6}(1)  (2017) p.~18.

\bibitem{GiPaCuIt15}
C.~Giusti, E.~Pastalkova, C.~Curto and V.~Itskov, Clique topology reveals
  intrinsic geometric structure in neural correlations, {\em Proc. Natl. Acad.
  Sci. U.S.A.} {\bf 112}(44)  (2015)  13455--13460.

\bibitem{ReNoScect17}
M.~W. Reimann, M.~Nolte, M.~Scolamiero, K.~Turner, R.~Perin, G.~Chindemi,
  P.~D{\l}otko, R.~Levi, K.~Hess and H.~Markram, Cliques of neurons bound into
  cavities provide a missing link between structure and function, {\em Front.
  Comput. Neurosci.} {\bf 11}  (2017) p.~48.

\bibitem{WuOtBa03}
S.~Wuchty, Z.~N. Oltvai and A.-L. Barab{\'a}si, Evolutionary conservation of
  motif constituents in the yeast protein interaction network, {\em Nat.
  Genet.} {\bf 35}(2)  (2003) p. 176.

\bibitem{BaCeLaetc21}
F.~Battiston, G.~Cencetti, I.~Iacopini, V.~Latora, M.~Lucas, A.~Patania, J.-G.
  Young and G.~Petri, Networks beyond pairwise interactions: Structure and
  dynamics, {\em Phys. Rep.} {\bf 907}  (2021)  1--68.

\bibitem{BiZi18}
G.~Bianconi and R.~M. Ziff, Topological percolation on hyperbolic simplicial
  complexes, {\em Phys. Rev. E} {\bf 98}(5)  (2018) p. 052308.

\bibitem{AlBadeMoPeLa21}
U.~Alvarez-Rodriguez, F.~Battiston, G.~F. de~Arruda, Y.~Moreno, M.~Perc and
  V.~Latora, Evolutionary dynamics of higher-order interactions in social
  networks, {\em Nat. Hum. Behav.} {\bf 5}  (2021)  586--595.

\bibitem{SkAr19}
P.~S. Skardal and A.~Arenas, Abrupt desynchronization and extensive
  multistability in globally coupled oscillator simplexes, {\em Phys. Rev.
  Lett.} {\bf 122}(24)  (2019) p. 248301.

\bibitem{GaDiGa21}
L.~Gambuzza, F.~Di~Patti, L.~Gallo, S.~Lepri, M.~Romance, R.~Criado, M.~Frasca,
  V.~Latora and S.~Boccaletti, Stability of synchronization in simplicial
  complexes, {\em Nat. Commun.} {\bf 12}(1)  (2021)  1--13.

\bibitem{MaGoAr20}
J.~T. Matamalas, S.~G{\'o}mez and A.~Arenas, Abrupt phase transition of
  epidemic spreading in simplicial complexes, {\em Phys. Rev. Research} {\bf
  2}(1)  (2020) p. 012049.

\bibitem{CoBi17}
O.~T. Courtney and G.~Bianconi, Weighted growing simplicial complexes, {\em
  Phys. Rev. E} {\bf 95}(6)  (2017) p. 062301.

\bibitem{PeBa18}
G.~Petri and A.~Barrat, Simplicial activity driven model, {\em Phys. Rev.
  Lett.} {\bf 121}(22)  (2018) p. 228301.

\bibitem{KoSeKh21}
K.~Kovalenko, I.~Sendi{\~n}a-Nadal, N.~Khalil, A.~Dainiak, D.~Musatov, A.~M.
  Raigorodskii, K.~Alfaro-Bittner, B.~Barzel and S.~Boccaletti, Growing
  scale-free simplices, {\em Commun. Phys.} {\bf 4}(1)  (2021)  1--9.

\bibitem{Ha02}
A.~Hatcher, {\em Algebraic Topology} (Cambridge University Press, 2002).

\bibitem{HaLa93}
T.~W. Haynes and L.~M. Lawson, Applications of {E}-graphs in network design,
  {\em Networks} {\bf 23}(5)  (1993)  473--479.

\bibitem{HaLa95}
T.~W. Haynes and L.~M. Lawson, Invariants of {E}-graphs, {\em Int. J. Comput.
  Math.} {\bf 55}(1-2)  (1995)  19--27.

\bibitem{WaYiXuZh21}
Y.~Wang, Y.~Yi, W.~Xu and Z.~Zhang, Modeling higher-order interactions in
  complex networks by edge product of graphs, {\em Comput. J.} {\bf 65}(9)
  (2022)  2347--2359.

\bibitem{ZhXuZhKaCh22}
M.~Zhu, W.~Xu, Z.~Zhang, H.~Kan and G.~Chen, Resistance distances in simplicial
  networks, {\em Comput. J. (in press)}   (2022).

\bibitem{DoGoMe02}
S.~N. Dorogovtsev, A.~V. Goltsev and J.~F.~F. Mendes, Pseudofractal scale-free
  web, {\em Phys. Rev. E} {\bf 65}(6)  (2002) p. 066122.

\bibitem{ZhZh21}
X.~Zhou and Z.~Zhang, Edge domination number and the number of minimum edge
  dominating sets in pseudofractal scale-free web and sierpi{\'n}ski gasket,
  {\em Fractals} {\bf 29}(07)  (2021) p. 2150209.

\bibitem{XiYu22}
C.~Xing and H.~Yuan, Lazy random walks on pseudofractal scale-free web with a
  perfect trap, {\em Fractals} {\bf 30}(1)  (2022) p. 2250030.

\bibitem{WaSlYuZh22}
X.~Wang, W.~Slamu, K.~Yu and Y.~Zhu, Maximum matchings in a pseudofractal
  scale-free web, {\em Fractals} {\bf 30}(4)  (2022) p. 2250077.

\bibitem{RoHaBe07}
H.~D. Rozenfeld, S.~Havlin and D.~Ben-Avraham, Fractal and transfractal
  recursive scale-free nets, {\em New J. Phys.} {\bf 9}(6)  (2007) p. 175.

\bibitem{ZhRoZh07}
Z.~Zhang, L.~Rong and S.~Zhou, A general geometric growth model for
  pseudofractal scale-free web, {\em Physica A} {\bf 377}(1)  (2007)  329--339.

\bibitem{ZhQiZhXiGu09}
Z.~Zhang, Y.~Qi, S.~Zhou, W.~Xie and J.~Guan, Exact solution for mean
  first-passage time on a pseudofractal scale-free web, {\em Phys. Rev. E} {\bf
  79}(2)  (2009) p. 021127.

\bibitem{ZhLiWuZh10}
Z.~Zhang, H.~Liu, B.~Wu and S.~Zhou, Enumeration of spanning trees in a
  pseudofractal scale-free web, {\em EPL (Europhys. Lett.)} {\bf 90}(6)  (2010)
  p. 68002.

\bibitem{PeAgZh15}
J.~Peng, E.~Agliari and Z.~Zhang, Exact calculations of first-passage
  properties on the pseudofractal scale-free web, {\em Chaos} {\bf 25}(7)
  (2015) p. 073118.

\bibitem{DiBoBe22}
C.~T. Diggans, E.~M. Bollt and D.~Ben-Avraham, Spanning trees of recursive
  scale-free graphs, {\em Phys. Rev. E} {\bf 105}(2)  (2022) p. 024312.

\bibitem{YiZhPa20}
Y.~Yi, Z.~Zhang and S.~Patterson, Scale-free loopy structure is resistant to
  noise in consensus dynamics in complex networks, {\em IEEE Trans. Cybern.}
  {\bf 50}(1)  (2020)  190--200.

\bibitem{XuWuZhZhKaCh22}
W.~Xu, B.~Wu, Z.~Zhang, Z.~Zhang, H.~Kan and G.~Chen, Coherence scaling of
  noisy second-order scale-free consensus networks, {\em IEEE Trans. Cybern.}
  {\bf 52}(7)  (2022)  5923--5934.

\bibitem{BaAl99}
A.-L. Barab{\'a}si and R.~Albert, Emergence of scaling in random networks, {\em
  Science} {\bf 286}(5439)  (1999)  509--512.

\bibitem{WaSt98}
D.~J. Watts and S.~H. Strogatz, Collective dynamics of `small-world' networks,
  {\em Nature} {\bf 393}(6684)  (1998)  440--442.

\bibitem{ShLiZh17}
L.~Shan, H.~Li and Z.~Zhang, Domination number and minimum dominating sets in
  pseudofractal scale-free web and sierpi{\'n}ski graph, {\em Theoret. Comput.
  Sci.} {\bf 677}  (2017)  12--30.

\bibitem{ShLiZh18}
L.~Shan, H.~Li and Z.~Zhang, Independence number and the number of maximum
  independent sets in pseudofractal scale-free web and sierpi{\'n}ski gasket,
  {\em Theoret. Comput. Sci.} {\bf 720}  (2018)  47--54.

\bibitem{Yu13}
R.~Yuster, Maximum matching in regular and almost regular graphs, {\em
  Algorithmica} {\bf 66}(1)  (2013)  87--92.

\bibitem{HoKlLiLiPoWa15}
W.-K. Hon, T.~Kloks, C.-H. Liu, H.-H. Liu, S.-H. Poon and Y.-L. Wang, On
  maximum independent set of categorical product and ultimate categorical
  ratios of graphs, {\em Theoret. Comput. Sci.} {\bf 588}  (2015)  81--95.

\bibitem{GaHaK15}
M.~Gast, M.~Hauptmann and M.~Karpinski, Inapproximability of dominating set on
  power law graphs, {\em Theoret. Comput. Sci.} {\bf 562}  (2015)  436--452.

\bibitem{CoLeLi15}
J.-F. Couturier, R.~Letourneur and M.~Liedloff, On the number of minimal
  dominating sets on some graph classes, {\em Theoret. Comput. Sci.} {\bf 562}
  (2015)  634--642.

\bibitem{Va79TCS}
L.~Valiant, The complexity of computing the permanent, {\em Theor. Comput.
  Sci.} {\bf 8}(2)  (1979)  189--201.

\bibitem{Va79SiamJComput}
L.~Valiant, The complexity of enumeration and reliability problems, {\em SIAM
  J. Comput.} {\bf 8}(3)  (1979)  410--421.

\bibitem{LoPl86}
L.~Lov{\'a}sz and M.~D. Plummer, {\em Matching Theory}, Annals of Discrete
  Mathematics, Vol.~29 (North Holland, New York, 1986).

\bibitem{NaAk12}
J.~C. Nacher and T.~Akutsu, Dominating scale-free networks with variable
  scaling exponent: heterogeneous networks are not difficult to control, {\em
  New J. Phys.} {\bf 14}(7)  (2012) p. 073005.

\bibitem{LiSlBa11}
Y.-Y. Liu, J.-J. Slotine and A.-L. Barab{\'a}si, Controllability of complex
  networks, {\em Nature} {\bf 473}(7346)  (2011)  167--173.

\bibitem{LiBa16}
Y.~Y. Liu and A.-L. Barab{\'a}si, Control principles of complex systems, {\em
  Rev. Mod. Phys.} {\bf 88}(3)  (2016) p. 035006.

\bibitem{LiLuYaXiWe15}
Y.~Liu, J.~Lu, H.~Yang, X.~Xiao and Z.~Wei, Towards maximum independent sets on
  massive graphs, {\em Proceedings of the 2015 International Conference on Very
  Large Data Base} {\bf 8}(13)  (2015)  2122--2133.

\bibitem{ChLiZh17}
L.~Chang, W.~Li and W.~Zhang, Computing a near-maximum independent set in
  linear time by reducing-peeling, in {\em Proceedings of the 2017 ACM
  International Conference on Management of Data\/} ACM2017, pp. 1181--1196.

\bibitem{HaPrVo99}
J.~Harant, A.~Pruchnewski and M.~Voigt, On dominating sets and independent sets
  of graphs, {\em Comb. Probab. Comput.} {\bf 8}(6)  (1999)  547--553.

\bibitem{WuDuJiLiHu06}
W.~Wu, H.~Du, X.~Jia, Y.~Li and S.~C.-H. Huang, Minimum connected dominating
  sets and maximal independent sets in unit disk graphs, {\em Theoret. Comput.
  Sci.} {\bf 352}(1-3)  (2006)  1--7.

\bibitem{Ge12}
G.~D. Birkhoff, A determinant formula for the number of ways of coloring a map,
  {\em Ann. Math.} {\bf 14}(1/4)  (1912)  42--46.

\bibitem{Bi93}
N.~Biggs, {\em Algebraic Graph Theory, 2nd edn.} (Cambridge University Press,
  1993).

\bibitem{St73}
R.~P. Stanley, Acyclic orientations of graphs, {\em Discrete Math.} {\bf 5}(2)
  (1973)  171--178.

\bibitem{GrZa83}
C.~Greene and T.~Zaslavsky, {On the interpretation of Whitney numbers through
  arrangements of hyperplanes, zonotopes, non-Radon partitions, and
  orientations of graphs}, {\em Trans. Amer. Math. Soc.} {\bf 280}(1)  (1983)
  97--126.

\bibitem{Tu54}
W.~T. Tutte, A contribution to the theory of chromatic polynomials, {\em Can.
  J. Math.} {\bf 6}(1)  (1954)  80--91.

\bibitem{We99}
D.~Welsh, {The Tutte polynomial}, {\em Random Struct. Alg.} {\bf 15}(3-4)
  (1999)  210--228.

\bibitem{DiHo98}
P.~W. Diaconis and S.~P. Holmes, Matchings and phylogenetic trees, {\em Proc.
  Natl. Acad. Sci.} {\bf 95}(25)  (1998)  14600--14602.

\bibitem{LiPaYiZh20}
H.~Li, S.~Patterson, Y.~Yi and Z.~Zhang, Maximizing the number of spanning
  trees in a connected graph, {\em IEEE Trans. Inf. Theory} {\bf 66}(2)  (2020)
   1248--1260.

\bibitem{Ca89}
A.~Cayley, A theorem on trees, {\em Quart. J. Math.} {\bf 23}  (1889)
  376--378.

\bibitem{DoMa05}
J.~P. Doye and C.~P. Massen, Self-similar disk packings as model spatial
  scale-free networks, {\em Phys. Rev. E} {\bf 71}(1)  (2005) p. 016128.

\bibitem{ZhZhZoChGu07}
Z.~Zhang, S.~Zhou, T.~Zou, L.~Chen and J.~Guan, Incompatibility networks as
  models of scale-free small-world graphs, {\em Eur. Phys. J. B} {\bf 60}(2)
  (2007)  259--264.

\bibitem{AgBu09}
E.~Agliari and R.~Burioni, Random walks on deterministic scale-free networks:
  Exact results, {\em Phys. Rev. E} {\bf 80}(3)  (2009) p. 031125.

\bibitem{ZhGuDiChZh09}
Z.~Zhang, J.~Guan, B.~Ding, L.~Chen and S.~Zhou, Contact graphs of disk
  packings as a model of spatial planar networks, {\em New J. Phys.} {\bf
  11}(8)  (2009) p. 083007.

\bibitem{CoFeRa04}
F.~Comellas, G.~Fertin and A.~Raspaud, Recursive graphs with small-world
  scale-free properties, {\em Phys. Rev. E} {\bf 69}(3)  (2004) p. 037104.

\bibitem{BaCoDaFi09}
L.~Barriere, F.~Comellas, C.~Dalf{\'o} and M.~A. Fiol, The hierarchical product
  of graphs, {\em Discrete Appl. Math.} {\bf 157}  (2009)  36--48.

\bibitem{QiYiZh19}
Y.~Qi, Y.~Yi and Z.~Zhang, Topological and spectral properties of small-world
  hierarchical graphs, {\em Comput. J.} {\bf 62}(5)  (2019)  769--784.

\bibitem{HiBe06}
M.~Hinczewski and A.~N. Berker, Inverted berezinskii-kosterlitz-thouless
  singularity and high-temperature algebraic order in an ising model on a
  scale-free hierarchical-lattice small-world network, {\em Phys. Rev. E} {\bf
  73}(6)  (2006) p. 066126.

\bibitem{ZhZhZoChGu09}
Z.~Zhang, S.~Zhou, T.~Zou, L.~Chen and J.~Guan, Different thresholds of bond
  percolation in scale-free networks with identical degree sequence, {\em Phys.
  Rev. E} {\bf 79}(3)  (2009) p. 031110.

\end{thebibliography}

\end{multicols}
\end{document}